\documentclass[a4paper, 11pt]{article}

\usepackage{hyperref}
\usepackage{srcltx}
\usepackage{indentfirst}
\usepackage{epsfig}
\usepackage{psfrag}
\usepackage{graphicx}
\usepackage{overpic}
\usepackage{multirow}
\usepackage[hang]{subfigure}
\usepackage{amsmath,amssymb, amsthm}
\usepackage[dvips]{color}
\usepackage[mathscr]{euscript}

\newtheorem{theorem}{Theorem}
\newtheorem{lemma}{Lemma}

\newtheorem{property}{Property}

\newcommand{\figref}[1]{\figurename~\ref{#1}}

\newcommand\bbR{\mathbb{R}}
\newcommand\bbZ{\mathbb{Z}}

\newcommand\bff{{\bf{f}}}

\newcommand\dd{\,\mathrm{d}}

\newcommand{\mB}{\mathcal{B}}

\newcommand{\Vw}{V_w}

\newcommand\BigNorm[1]{ \left \| #1 \right \| } 
\newcommand\pd[2]{\dfrac{\partial {#1}}{\partial {#2}}}
\newcommand\opd[2]{\dfrac{\dd {#1}}{\dd {#2}}}

\newcommand\ri{{\rm{i}}} 

\newcommand\comment[1]{}

\graphicspath{{images/}}
{\theoremstyle{remark} }

\begin{document}

\title{Singularity-free Numerical Scheme  for the  Stationary Wigner Equation}

\author{
  Tiao Lu\thanks{CAPT, HEDPS, LMAM,
 IFSA Collaborative Innovation Center of MoE, \&
    School of Mathematical Sciences, Peking University, Beijing,
    China, email: {\tt tlu@math.pku.edu.cn}.},
    ~~ Zhangpeng Sun \thanks{School of Mathematical Sciences, 
	  Peking University, Beijing, China, 
	  email: {\tt sunzhangpeng@pku.edu.cn}.}
}
\maketitle

\begin{abstract}

For the stationary Wigner equation with inflow boundary conditions,
its numerical  convergence with respect to the velocity mesh size 
are deteriorated due to the singularity at velocity zero. 
In this paper, using the fact that the solution of the stationary Wigner equation
 is subject to an algebraic constraint, we prove that the Wigner equation can 
 be written into a form with a bounded operator $\mathcal{B}[V]$,
 which is equivalent to the operator $\mathcal{A}[V]=\Theta[V]/v$ 
 in the original Wigner equation under some conditions. 
 Then the discrete operators discretizing $\mathcal{B}[V]$ are proved to 
 be uniformly  bounded with respect to the mesh size.
 Based on the therectical findings, a signularity-free numerical method is proposed.
 Numerical reuslts are proivded to show our improved numerical scheme performs
 much better in 
 numerical convergence than the  original scheme based on discretizing $\mathcal{A}[V]$. 

\vspace*{4mm}
\end{abstract}

\section{Introduction}

The Wigner transport equation is one of the quantum mechanical
frameworks. It is proposed by E. Wigner in 1932 as a quantum correction
to the classical statistical mechanics{\cite{Wigner1932}}. 
Though the Wigner function may take negative values, it has a
non-negative marginal distribution and can express system observables
in the same way as
the Boltzmann probability density function, thus it is called a  
quasi-probability density function. The strong similarity between the Wigner 
equation and the Boltzmann equation makes it convenient to borrow 
some describing tools of the latter, e.g., the boundary 
conditions and the scattering terms{\cite{ferry_transport}}.

The Wigner equation has been used in many fields. For example, 
Frensley successfully reproduced the negative differential
resistance phenomena of resonant tunneling devices by numerically
solving the following one-dimensional Wigner equation
\begin{equation}
  \frac{\partial f}{\partial t}+ v\frac{\partial f}{\partial x} - \Theta[V]f = 0 , x \in(-l/2,l/2), v \in
  \bbR , 
\end{equation}
with inflow boundary conditions
\begin{equation}
  f(-l/2,v) = f_l(v),  \text{ if } v>0;   f(l/2,v)=f_r(v), \text{ if } v<0.
\end{equation}
$\Theta[V]$ is a pseudo-differential operator that will be
explained later. Since then, the Wigner equation 
has attracted many researchers in numerical simulation (e.g.,
\cite{KoNe06} and references therein), and 
various numerical methods for the Wigner equation have been
proposed, 
such as finite difference methods
\cite{Frensley1987,Jensen1990,Biegel1996, Jiang2011}, 
spectral methods \cite{ShLu09, Ringhofer1990,Biegel1996},
spectral element method \cite{ShLu09}, 
and Monte Carlo methods{\cite{nedialkov_04,Sellier20142427}. 
When the Hartree potential is included, the
Wigner-Poisson system self-consistently can be solved
\cite{MarkowichDegond1990,Jensen1991,Arnold1996,Zhao03,Manzini2005}. 
The nonlinear iteration for the coupled Wigner-Poisson system 
deserves a serious study and in \cite{Biegel1996}  the Gummel method
and the 
Newton method were compared for the RTD simulation in terms of 
efficiency, accuracy and robustness. As for the linear stationary
Wigner equation with inflow boundary boundary conditions, there are
still a lot of open problems, for example, the well-posedness, the
numerical convergence, etc. In this paper, we focus on the linear
problem. 

Many mathematicians have been drawn to study the Wigner equation, e.g., 
\cite{Neunzert1985,Markowich1989,goudon02,Golo02,Morandi2012}. 
The existence and uniqueness of the solution for the Wigner equation 
have been proved by using of the theoretical result of the Schr\"odinger equation. 
But there are still a lot of open problems. One of them is to 
build the well-posedness result for the Wigner boundary value problem
(the stationary Wigner equation with inflow boundary conditions), 
which is a popular model in numerical simulation of the nano 
semiconductor devices. We note that some researchers have proved the 
well-posedness of the Wigner boundary value problem in some special
cases, for example, \cite{ALZ00} for a velocity semi-discretization
version, \cite{Zweifel2001} for an approximate problem by removing a 
small interval centered at velocity zero, and \cite{LiLuSun2014} for a
periodical potential.  

It is pointed out that the Wigner BVP problem is more difficult than
the Wigner initial value problem, and one of the reason is that the
inflow boundary conditions break up the equivalence between the Wigner
equation and the Schr\"odinger equation. When using a numerical
methods to discretize the Wigner equation, computational parameters
such the mesh size and the correlation length are sensitive and need a
careful calibration \cite{Zhao03}. In \cite{Kim1999}, several
numerical schemes including first-order (FDS) and second-order difference 
schemes (SDS) for discretization of the advection term 
$v\frac{\partial f}{\partial x}$ of the Wigner equation were compared. 
However, to authors' knowledge, a detailed accuracy study of the
finite difference methods for the Wigner BVP with respect to the velocity mesh 
size has not been reported.  One of the difficulties is that the operator 
$\frac{1}{v}\Theta[V]$ is singular at $v=0$, which results in the
numerical solution's oscillation and blowing up as the velocity mesh size
goes to zero.  Assuming that the Wigner BVP has a unique smooth
solution, we observe that the solution must satisfy 
$(\Theta[V]f)(x,0) =0$. In this paper, we design a new numerical 
scheme by applying this constraint in our numerical scheme. 

The rest of the paper is arranged as follows. In Section 2, we rewrite the original 
Wigner equaiton into a form with a bounded operator
 $\mathcal{B}[V]$ which is equivalent to $\mathcal{A}[V]$
 under the assumption that the  distribution function satsfies an algebraic constraint.  
In Section 3, we prove the discrete operators discretizing $\mathcal{B}[V]$ to 
 be uniformly  bounded with respect to the mesh size. 
 Based on this analysis,  a new numerical method is proposed. 
At last, in Section 4,  we give some numerical examples to show the numerical convergence
with respect to $x$-space and $v$-space.  Some clonclusion remarks are given in last section.

\section{Stationary Wigner equation and an equivalent form}
We are concerned with the following stationary Wigner equation (or
"quantum Liouville equation ")
\begin{equation} \label{eq:Wigner}
 v \frac{\partial f}{\partial x} = \Theta[V]f(x,v),
\end{equation}
where $\Theta[V]$ is a pseudo-differential operator defined by
\begin{equation}
\Theta[V]f(x,v)=\int_{- \infty}^{\infty} 
 V_w \left( x, v - v' \right) f 
 \left( x, v' \right) d v' .
\end{equation}
$V_w(x,v)$ is called the Wigner potential and defined by 
\begin{equation} \label{eq:OriginalVw}
V_w \left( x, v \right) = \frac{i}{2 \pi} \int_{- \infty}^{\infty} 
D_V(x,y)
\exp \left( i v y \right) \mathrm{dy}, 
\end{equation}
where 
\[
D_V(x,y) = V \left( x + y / 2 \right) - V \left( x - y / 2 \right) 
\]
is the difference of the potential $V$ at positions $x+y/2$ and
$x-y/2$. $\Theta[V]$ can be proved to be a a continuous (bounded)
linear operator on $L^2(\mathbb{R})$ if $V\in L^{\infty}(\bbR)$.

In this paper, we use the Fourier transform and its
corresponding inverse defined as
\begin{equation}\label{eq:Fourier}
 \hat f (y) = \mathcal{F}(f) =\int_{-\infty}^\infty f(v) \exp(-i v y) \dd v,
\end{equation}
\begin{equation}\label{eq:Fourier inverse}
 f (v) = \mathcal{F}^{-1}(\hat f) =\frac{1}{2\pi}\int_{-\infty}^\infty
\hat f(y)  \exp(i v y) \dd y. 
\end{equation}
In the sense of Fourier transform, $\Theta[V]$ can also be written as 
\begin{equation}
 \Theta[V] f(v)=V_w \ast f (v) = i \mathcal{F}^{-1}(D_V) \ast 
 \mathcal{F}^{-1}(\hat f) =i\mathcal{F}^{-1}(D_V \hat f),
\end{equation}
where $\ast$ denotes convolution.
By the Parseval equality, we have
\begin{equation}
 \|\Theta[V]f\|_2 =\frac{1}{2\pi} \|D_V \hat f\|_2
 \leqslant 2\cdot \frac{1}{2\pi} \|V\|_\infty \cdot \|\hat f\|_2
 =2 \|V\|_\infty \cdot \| f\|_2.
\end{equation}
Immediately, we have the following lemma.
\begin{lemma}\label{lemma:bound}
Suppose that the potential $V \in L^{\infty}(\bbR)$. For any
$x\in\bbR$, the operator $\Theta[V]$ is
a bounded linear operator on $L^2(\bbR_v)$, and
\[
\| \Theta[V] \| \leqslant 2 \| V \|_{\infty}.
\]
\end{lemma}

Dividing \eqref{eq:Wigner} by $v$ gives
\begin{equation} \label{eq:evolution}
 \pd{f}{x}=\mathcal A[V]f(x,v) 
\end{equation}
where the operator $\mathcal A[V]$ is defined by 
\begin{equation}
\mathcal A[V]f(x,v)=\frac{1}{v} \Theta[V] f(x,v)= \dfrac{1}{v} \int_{-\infty}^\infty V_w(x,v-v^\prime) f(x,v^\prime) \dd v^\prime.
\end{equation}
The equation \eqref{eq:evolution} can be viewed as an evolution system,
which gives us a convenient way to analyze and 
compute the stationary Wigner equation. However,
$\frac{V_w(x,v-v')}{v}$, the kernel of $\mathcal{A}[V]$,  is singular
at $v=0$, and this brings great difficulty to solve and analyze 
\eqref{eq:evolution}.
Although we usually avoid the point at "$v=0$" to be a mesh
point in numerical experiments,  the numerical distribution can
suffer from severe oscillation when using a small velocity mesh size.
It is a reason that no numerical convergence work has been published. 

Under the condition that $f(x,v)$ is Lipschitz continuous with respect
to $x$ which ensures the boundedness of $\pd{f(x,v)}{x}$, setting  $v
= 0$ in \eqref{eq:Wigner} yields
\begin{equation}\label{eq:V0}
\int V_w \left( x, - v' \right) f \left( x, v' \right) d v' = 0.
\end{equation}
This is an important property of the stationary Wigner equation,
which will be used to design a numerical method. 
Subtracting \eqref{eq:V0} from \eqref{eq:Wigner}, we obtain 
\begin{equation}\label{eq:ImprovedWigner}
v \frac{\partial f}{\partial x} = 
\int \left[ V_w \left( x, v - v' \right)
   - V_w \left( x, v \right) \right] f \left( x, v' \right) d v'. 
\end{equation}
Then we divide the above equation by $v$, and obtain
\begin{equation} \label{eq:ImprovedWignerB}
\frac{\partial f}{\partial x} = \mathcal{B}[V] f(x,v),
\end{equation}
where the operator $ \mathcal{B}[V] $ is defined as 
\begin{equation}\label{eq:BV}
\mathcal{B}[V] f(x,v)=\int \frac{V_w
   \left( x, v - v' \right) - V_w \left( x, - v' \right)}{v} f(x,v') d v' .
\end{equation}
Equation \eqref{eq:ImprovedWignerB}-\eqref{eq:BV} is equivalent to the stationary
Wigner equation if \eqref{eq:V0} holds.

Now we focus on the properties of the operator $\mathcal B[V]$. It is 
evident that the operator $\mathcal B[V]$ is different 
from $\mathcal A[V]$. But they are equal on some special spaces, 
e.g., $f(x,v)$ is an even function with respect to $v$. We will prove 
that $\mathcal{B}[V]$ is a bounded linear operator under some 
assumptions. The details are shown in the following Theorem
\ref{Thm:BoundofB}. 

\begin{theorem} \label{Thm:BoundofB}
 Suppose that for any $x \in [-l/2,l/2]$, the Wigner
potential $V_w(x,\cdot)$ defined in \eqref{eq:OriginalVw} 
belongs to $H^1(\bbR_v)$. Then $\mathcal
B[V]: 
 L^2(\bbR_v) \to L^2(\bbR_v)$ is a bounded
 linear operator.
\end{theorem}
\begin{proof}
For any $f \in L^2(\bbR_v)$, 
 \begin{equation} \label{eq:L2normB}
  \|\mathcal{B}[V] f\|_{L^2(\bbR_v)}^2 =\int_{-\infty}^\infty |\mathcal{B}[V] f|^2 \dd v.
 \end{equation}
We will prove the boundedness of $\mathcal{B}[V]$ by estimating \eqref{eq:L2normB}
 on the region $|v|>1$ and the region $|v|\leqslant 1$ respectively.

 First, we consider the part with $\left| v \right| > 1$.  Using $V_w
  \left( x, v \right) \in L^2 \left( \mathbb{R}_v \right)$ and the
  Young's inequality, we have
  \begin{equation} \label{eq:YoungInequality}
    \left\|\Theta[V]  f( x, v )\right\|_{L^\infty(\bbR_v)}
     =\|V_w \left( x, v \right) \ast f \left( v
    \right) \|_{L^{\infty} \left( \mathbb{R}_v \right)} \leq \|V_w
    \left( v \right) \|_{L^2(\bbR_v)} \|f \left( v \right) \|_{L^2
      \left( \mathbb{R}_v \right)}. 
  \end{equation} 
  By the Cauchy-Schwartz inequality, we then have  
  \begin{equation} \label{eq:CauchyInequality}
    \left|
      \int_{\mathbb{R}_{v'}} V_w \left( x, 0 - v' \right) f \left( x,
        v' \right) \dd v' \right| \leq \|V_w \left( x, v \right)
    \|_{L^2(\bbR_v)} \|f \left( x, v \right) \|_{L^2 \left(
        \mathbb{R}_{v} \right)} .  
  \end{equation}
  It is obtained directly from \eqref{eq:YoungInequality} and 
  \eqref{eq:CauchyInequality} that 
  \begin{align}
    & \qquad \int_{\left| v \right| > 1} \left|\mathcal{B}[V]f(x,v) \right|^2 \dd v \notag \\
    & \leqslant
    2 \int_{\left| v \right| > 1} \left| \frac{\Theta[V]
        f \left( x, v \right)}{v} \right|^2 \dd v +2  \int_{\left| v
      \right| > 1} \frac{\|V_w \left( x, \cdot \right) \|_{L^2(\bbR_v)}^2 \|f
      \left( x, \cdot \right) \|_{L^2 \left( \mathbb{R}_{v}
        \right)}^2}{v^2} \dd v  \notag \\ 
    & \leqslant 8\|V_w
    \left( x, v \right) \|^2_{L^2(\bbR_v)} \|f \left( x, v
    \right) \|^2_{L^2(\bbR_v)} . \label{eq:firstPart}
  \end{align}

  Then, we consider the part with $\left| v \right| \leqslant
  1$. According to the Cauchy-Schwartz inequality again, we have
  \begin{align*}
    \left|\mathcal{B}[V] f(x,v) \right|   & \leqslant
    \int_{\mathbb{R}} \left| \frac{V_w \left( x, v - v'
	\right) - V_w \left( x, 0 - v' \right)}{v} \right| \left| f \left(
        x, v' \right) \right| \dd v' \\
    & \leqslant \left \| \frac{V_w \left( x, v - v' \right) - V_w \left( x,
          0 - v' \right)}{v} \right\|_{L^2 \left( \mathbb{R}_{v'} \right)} \|f
    \left( x, \cdot \right) \|_{L^2(\bbR_v)}, \quad v \in \left[ - 1,
1 \right] .
  \end{align*}
  By using Theorem 3 in Chapter 5 of \cite{Evans2010}, we have
  \[
  \BigNorm{ \frac{\Vw \left( x, v_{} - v'
      \right) - \Vw \left( x, - v' \right)}{v} }_{L^2 \left(
      \mathbb{R}_{v'} \right)} \leqslant 
  \BigNorm{ \partial_{v'} \Vw \left( x,
      v' \right) }_{L^2 \left( \mathbb{R}_{v'} \right)} .
  \]
This fact, together with the Cauchy-Schwartz inequality,
  gives us the following estimate on the velocity interval $[-1,1]$
  that
  \begin{equation} \label{eq:secondPart}
    \int_{\left| v \right|
      \leq 1} \left|\mathcal{B}[V]f(x,v) \right|^2 \dd v \leq 
    \|f \left( x, v \right) \|^2_{L^2(\bbR_v)}   \BigNorm{ \partial_{v}
\Vw \left( x,
      v \right) }^2_{L^2 \left( \mathbb{R}_{v} \right)} 
  \end{equation} 
  Collecting \eqref{eq:firstPart} and \eqref{eq:secondPart} together results in
  \[
  \| \mathcal{B}[V] f(x,v) \|^2_{L^2(\bbR_v)}
  \leqslant C \|f \left( x, v \right) \|^{^2}_{L^2(\bbR_v)} 
  \]
  where 
  \[ C = 8 \BigNorm{\Vw \left( x, v
    \right) } ^{^2}_{H^1(\bbR_v)} . 
  \]
This completes the proof that $\mathcal{B}[V]$ is a  bounded linear
  operator on $L^2(\bbR_v)$.
  
\end{proof}


We have proved the operator  $\mathcal{B}[V]$ is bounded, thus obtained 
a singularity-free form \eqref{eq:ImprovedWignerB}-\eqref{eq:BV}, which is
euquivalent to the original Wigner equation. A numerical scheme based on 
the singularity-free from will be porposed in the next section.

\section{Singularity-free numerical scheme}
 We start from  the discretization of the
pseudo-differential operator $\Theta[V]$, then define the corresponding discrete operators of $\mathcal{A}[V]=\frac{1}{v}\Theta[V]$
and $\mathcal{B}[V]$, respectively. At the end of this section, 
we prove that the discrete operators of $\mathcal{B}[V]$ is uniformly bounded with respect to the velocity mesh size.

Before introducing the discretization of the pseudo-differential
operator $\Theta[V]$, 
we define a new operator $\Theta^h[V]: L^2(\bbR_v) \to L^2(\bbR_v)$, the approximation of $\Theta[V]$, 
\begin{eqnarray}
  \Theta^h[V](f) &  =& i \mathcal{F}^{-1}_{y\rightarrow v} \left(D_V
	\chi_{|y|\leqslant R^h } \hat{f}(y) \right),  \qquad \forall f \in L^2(\bbR_v),
\end{eqnarray}
where $\chi_{|y|\leqslant R^h}$ is the characteristic function of
$\{y: |y|\leqslant R^h\}$.
Here $h$ is related to the velocity mesh size ($\Delta v=2\pi h$), and $R^h$ is
related to $h$ by
\begin{equation}\label{Rhdef}
R^h = \frac{1}{2h}.
\end{equation}
 In some papers, e.g.,
\cite{Frensley1987}, $R^h$ is called the coherence length.
We introduce a subspace $L^2_h(\bbR)$ of $L^2(\bbR)$ defined as 
\begin{equation}
L^2_h(\bbR) = \left\{ g \in L^2(\bbR) | \text{supp} \hat{g} \subset
[-R^h,R^h]\right\}.
\end{equation}
For any function $f^h \in L^2_h(\bbR)$, by using the
Shannon sampling theorem, we have 
\begin{equation}\label{fhs}
 f^h(v)= \sum_{n=-\infty}^\infty f_n \text{sinc} \left(
\frac{v-v_n}{2h} \right),
\end{equation}
where 
\begin{equation}
\text{sinc}(x) := \dfrac{\sin x}{x},
\end{equation}
and 
\begin{equation} \label{eq:fnvn}
f_n=f^h(v_n),  v_n =(2n+1)\pi h ,\ n\in \bbZ .
\end{equation}
$ \{v_n : n\in \bbZ\}$are the sampling velocity
points, and the series is absolutely and uniformly convergent on
compact sets \cite{Zayed1993}. Actually,
\begin{equation}
 \int_{\bbR_v}
\text{sinc}(\frac{v-v_n}{2h})\text{sinc}(\frac{v-v_m}{2h}) \dd v
 = \begin{cases} 2 \pi h, & \text{if } n=m, \\
    0, & \text{else,}
   \end{cases}
\end{equation}
and 
$\left\{ \text{sinc}(\frac{v-v_n}{2h}) : n\in \bbZ \right\}$ is
an orthogonal basis of $L^2_h(\bbR)$.  
From \eqref{fhs},
we can define an isometry (disregarding a constant) 
$\mathcal{I}_h: L^2_h(\bbR) \rightarrow l^2(\bbZ)$: for any $f^h\in
L^2_h(\bbR)$, 
\begin{equation} \label{isometry}
 \mathcal{I}_h f^h = (\cdots, f_{-1}, f_{0},f_{1}, \cdots)^T 
\end{equation}
where $\left\{ f_n\right\}$ is defined in \eqref{eq:fnvn}. It
is easy to see that 
\begin{equation}
 \| \mathcal{I}_h f^h \|_{l^2(\bbZ)} = \frac{1}{\sqrt{ 2\pi h}}
\|f^h\|_{L^2(\bbR)}. 
\end{equation}

$\Theta^h[V]$ can be considered as the
restriction of $\Theta[V]$ on $L^2_h(\bbR)$, and 
there are some obvious properties for the operator $\Theta^h[V]$,
showing in Property 
\ref{Prop: Thetah}. 
\begin{property} \label{Prop: Thetah}
The approximated operator $\Theta^h[V]$ fulfills the following
  properties:
\begin{enumerate}
\item[(i)] if $f\in L^2_h(\bbR)$, 
  then $\Theta^h[V](f) = \Theta[V](f)$;
\item[(ii)] if $f\in L^2(\bbR)$, then $\Theta^h[V](f)$ converges to
$\Theta[V](f)$ in $L^2(\bbR)$ as $h
\rightarrow 0$. If, furthermore, $f$ lies in the Soblev space
$H^s(\bbR)$, $s>0$, we get 
\[
\| \Theta^h[V](f) - \Theta[V](f) \|^2_{L^2(\bbR)}
\leqslant \dfrac{4 \|V\|_{\infty}}{2\pi} \dfrac{\| f
  \|^2_{H^s_{\bbR}}}{\left( 1+(R^h)\right)^s}. 
\]
\end{enumerate}
\end{property}

We consider the discretization of $\Theta[V]f$ in the Wigner equation
for $f(x,v)$ assuming that $f(x,v) \in L^2(\bbR), \forall x \in
[-l/2,l/2]$. We use $f_n(x)$ to represent the numerical approximation of
$f(x,v_n)$, and $\bff = \{f_n: n\in \bbZ\} \in l^2(\bbZ) $. Based on
Property \ref{Prop: Thetah} and the Shannon
sampling theorem, a discrete operator
$\Theta_d[V]:l^2(\bbZ)\rightarrow
l^2(\bbZ)$ as an approximation of $\Theta[V]$ on $L^2(\bbR)$ can be
constructed by  
\begin{equation} \label{thetaddef}
\Theta_d[V] \bff = \mathcal{I}_h  \Theta^h[V] \mathcal{I}^{-1}_h \bff=
2\pi h M^{\Theta_d} \bff \end{equation}
where $\mathcal{I}_h$ is defined in \eqref{isometry},
$\mathcal{I}_h^{-1}$ is the inverse of $\mathcal{I}_h$, 
and $M^{\Theta_d}$ is an infinite dimensional matrix with 
\begin{equation} \label{eq:MTheta_d}
M^{\Theta_d}_{nm} = \frac{\ri}{2\pi}\int_{\bbR} D_V(x,y)
\chi_{(-R^h,R^h)}(y) 
e^{\ri (v_n-v_m) y } \dd y.
\end{equation}
$M^{\Theta_d}$ is the matrix of a discrete
convolution operator and $M^{\Theta_d}_{nm}$ depends only on $n-m$. 
And $M^{\Theta_d}$ is a real-valued skew-symmetric matrix,
i.e., $M^{\Theta_d}_{nm} = - M^{\Theta_d}_{mn}$. 

We are able to establish a property for the operator $\Theta_d[V]$,
which is the discrete analogue
of Lemma \ref{lemma:bound}.
\begin{property}
 The operator $\Theta_d$ is a bounded linear operator on $\ell^2$, and
its norm is estimated uniformly with respect to $h$ by $\| \Theta_d
\|_{\mathcal{L}(\ell^2)} \leqslant 2 \| V \|_{\infty}$.
\end{property}

A typical semi-discretization of the orignal stationary Wigner equation with
inflow boundary conditions can be written as 
\begin{equation}\label{eq:DiscreteProblem}
\left\{
\begin{array}{ll}
 \opd{\bff}{x} =
 \mathcal{A}_d[V] \bff , & x \in (-l/2,l/2),  \\
f_{n}(-l/2) = f_L(v_n) , & \mathrm{if} \ v_n>0,  \\
f_{n+1/2}(l/2) = f_R( v_n), & \mathrm{if} \ v_n<0,  \\
\end{array}
\right .
\end{equation}
where the operator
$\mathcal{A}_d[V]: l^2\rightarrow l^2$ is
defined 
by 
\[
(A_d[V]\bff)_n = \frac{1}{v_n} (\Theta_d[V]\bff)_n,
\]
where $\Theta_d[V]$ is defined in \eqref{thetaddef}.
$A_d[V]$ is obviously bounded, but its norm will grow to infinity as
the velocity mesh size $h \rightarrow 0$. The property makes the numerical
solution suffer from numerical instability when a small velocity mesh
size $h$ is used, and it also affects the numerical convergence of
the numerical solution. 

Based on the equivalent singularity-free form 
\eqref{eq:ImprovedWignerB}-\eqref{eq:BV}, we can derive 
a semi-discretization scheme for the stationary Wigner equation with inflow boundary conditions
\begin{equation}\label{eq:NewDiscreteProblem}
\left\{
\begin{array}{ll}
 \opd{\bff}{x} =
 \mathcal{B}_d[V] \bff , & x \in (-l/2,l/2),  \\
f_{n}(-l/2) = f_L(v_n) , & \mathrm{if} \ v_n>0,  \\
f_{n+1/2}(l/2) = f_R( v_n), & \mathrm{if} \ v_n<0,  \\
\end{array}
\right .
\end{equation}
Here the discrete operator $\mB_d[V] :l^2\rightarrow l^2$ is obtained by 
discretization of the bounded operator $\mathcal{B}[V]$, and can be written 
out as 
\begin{equation} \label{mbd}
\left(\mB_d[V]\bff \right)_ { n }
= \frac{2\pi h }{v_n} \sum_{m\in \bbZ}
(M^{\Theta_d}_{nm}-a_m )  f_m,
\end{equation}
where $M^{\Theta_d}_{nm}$ is given in \eqref{eq:MTheta_d} and 
\begin{equation} \label{eq:a_m}
a_{m} =\frac{\ri}{2\pi}\int_{\bbR} D_V(x,y) \chi_{(-R^h,R^h)}(y)
e^{-\ri
v_m y} \dd y.
\end{equation}

We will prove $\mB_d[V]$ is uniformly bounded with some assumptions of 
the potential $V$ in the following theorem.
\begin{theorem}
Suppose that $V_w(x,\cdot) \in H^1(\bbR_v)$ for any $x\in [-l/2,l/2]$. 
For a given velocity mesh size $h>0$, we define $\mB_d[V]
:l^2(\bbZ)\rightarrow l^2(\bbZ)$ as in \eqref{mbd} where
$v_n=(2n+1)\pi h$. 
Then $\mB_d[V]$ is uniformly bounded i.e., $\| \mB_d[V]\| \leq C$
where $C$ does not depend on the velocity mesh size $h$ or $x$.  
\end{theorem}
\begin{proof}
The proof is similar to that of Theorem 1. First,  we consder  $|v_n| \leqslant 1 $.
For $\mathbf{g} = \{g_n:n\in\bbZ\} \in
l^2(\bbZ)$,  we can write  $n$th-component of $\mB_d[V] \mathbf{g}$ as
\begin{equation}\label{mbdvn}
 ( \mB_d[V] \mathbf{g} )_{n} =  \ri h \sum_{m\in\bbZ} g_m
 \int_{-\frac{1}{2h}}^{\frac{1}{2h}} D_V(x,y) \frac{e^{\ri (v_n-v_m) y }
- e^{ -\ri v_m y } }{v_n} \dd y .
\end{equation}
by using \eqref{mbd}, \eqref{eq:MTheta_d} and \eqref{Rhdef}. 

For each $v_n$, we introduce a vector   $\mathbf{q}^{v_n}=\left\{q_m^{v_n}:m\in\mathbb{Z}\right\}$ defined as 
\begin{equation}\label{qmvn}
q_m^{v_n} = -h
\int_{-\frac{1}{2h}}^{\frac{1}{2h}}  y D_V(x,y)  \text{sinc} \frac{v_n y}{2}
  e^{\ri(\frac{v_n}{2}- v_m) y } \dd y.
\end{equation}
So $( \mB_d[V] \mathbf{g} )_{n}$ in \eqref{mbdvn} can be written as the
 $l^2(\bbZ)$ inner product of $\mathbf{q}^{v_n}$ and $\mathbf{g}$, i.e., 
\begin{equation}\label{gq}
 ( \mB_d[V] \mathbf{g} )_{n} =\left (
\mathbf{q}^{v_n},\mathbf{g}\right)_{l^2(\bbZ)}.
\end{equation}

Applying the Cauchy-Schwartz inequality to \eqref{gq}, we have 
\begin{equation} \label{zm1}
 | ( \mB_d[V] \mathbf{g} )_{n}  | \leqslant
\|\mathbf{q}^{v_n}\|_{l^2(\bbZ)} \|\mathbf{g}\|_{l^2(\bbZ)}. 
\end{equation}
By using the Parseval theorem on \eqref{qmvn}, we have  
\begin{equation}\label{zm2}
\begin{split} 
\sum_{m\in\bbZ}|q_m^{v_n}|^2  &  =h\int_{-\frac{1}{2h}}^{\frac{1}{2h}}  |
\chi_{(-R^h,R^h)}(y)   y  D_V(x,y) 
\text{sinc} \frac{v_n y}{2}
  |^2 \dd y  \\
  & \leqslant    h \int_{\bbR} |yD_V(x,y)|^2 \dd y
\leqslant
2\pi h \|V_w(x,\cdot)\|_{H^1(\bbR_v)}^2.
\end{split}
\end{equation}
The last $\leqslant$ is obtained by using the definition of 
$V_w(x,v)$ in \eqref{eq:OriginalVw}. 
From \eqref{zm1} and \eqref{zm2}, we conclude that 
$\{ | ( \mB_d[V] \mathbf{g} )_{n}  |:n\in\bbZ \}$ is bounded, i.e., 
\begin{equation} \label{zm3}
 | ( \mB_d[V] \mathbf{g} )_{n}  | \leqslant \sqrt{ 2\pi h}
  \|V_w(x,\cdot)\|_{H^1(\bbR_v)} \|\mathbf{g}\|_{l^2(\bbZ)}, \forall n \in \bbZ. 
\end{equation}
The number of $n$ such that $|v_n|\leqslant 1$ is less than $\frac{1}{\pi h}$, 
so we have 
\begin{equation}\label{xj1}
\sum_{ \{ n\in \bbZ :| (2n+1)\pi h| \leqslant 1\} }| ( \mB_d[V]
\mathbf{g} )_{n} | ^2 \leqslant 2
  \|V_w(x,\cdot)\|_{H^1(\bbR_v)}^2 \|\mathbf{g}\|_{l^2(\bbZ)}^2 . 
\end{equation}

Then we consider the case $|v_n| > 1$. Different from 
\eqref{gq}, we rewrite \eqref{mbdvn} into 
\begin{equation} \label{mbdvg1}
  ( \mB_d[V] \mathbf{g} )_{n} = \frac{1
}{v_n}(\mathbf{\tilde{q}}^{v_n},\mathbf{g})_{l^2(\bbZ)},
\end{equation}
where $\mathbf{\tilde{q}}^{v_n}=\{\tilde{q}_m^{v_n}:m\in\bbZ\}$ and 
\begin{equation} \label{qmvn1}
 \tilde{q}_m^{v_n} = \ri h
 \int_{\bbR} D_V(x,y) \chi_{(-R^h,R^h)}(y) (e^{\ri v_n y}-1)
   e^{-\ri v_m y } \dd y.
\end{equation}
Applying the Cauchy-Schwartz inequality to \eqref{mbdvg1}, we have 
\begin{equation} \label{tmpqm}
 |  ( \mB_d[V] \mathbf{g} )_{n}| \leqslant \frac{1}{v_n}
\|\mathbf{\tilde{q}}^{v_n}\|_{l^2(\bbZ)} \| \mathbf{g}\|_{l^2(\bbZ)}.
\end{equation}
By using the Parseval theorem on \eqref{qmvn1}, we have 
\begin{equation}\label{tildeqe}
\begin{split}
 \sum_{m\in\bbZ} |\tilde{q}_m^{v_n}|^2
  & = h \int_{\bbR}| D_V(x,y) \chi_{(-R^h,R^h)}(y)
(e^{\ri v_n y}-1)
   |^2 \dd y  \\
  & \leqslant 4 h  \int_{\bbR} |D_V(x,y)|^2 \dd y  \\
  & =4  h  \|D_V(x,\cdot)\|_{L^2(\bbR)}^2  \\
  & = 8\pi h  \| V_w(x,\cdot)\|_{L^2(\bbR)}^2 .
 \end{split} 
\end{equation}
Plugging \eqref{tildeqe} into \eqref{tmpqm} yields the estimate
\begin{equation}
 |  ( \mB_d[V] \mathbf{g} )_{n}| \leqslant \frac{2\sqrt{2 \pi h}}{v_n
}
 \| V_w(x,\cdot)\|_{L^2(\bbR)} \| \mathbf{g}\|_{l^2(\bbZ)}.
\end{equation}
Recalling the relation between $\sum_{n \in\bbZ, |v_n| >1}\frac{2\pi
h}{|v_n|^2}$ and $\int_{|v|>1}\frac{1}{v^2}\dd v$, we know that 
there exist a constant $C_2$ which does not depend on $h$ such that 
\begin{equation}\label{xj2}
\sum_{ n \in \{ n\in \bbZ : (2n+1)\pi h > 1\} }| ( \mB_d[V]
\mathbf{g} )_{n} | ^2 \leqslant C_2
  \|V_w(x,\cdot)\|_{L^2(\bbR_v)}^2 
  \| \mathbf{g}\|_{l^2(\bbZ)}^2.  
\end{equation}

Putting \eqref{xj1} and \eqref{xj2} together, we come to a conclusion
that 
there exists a constant $C$ which does not dependent on 
$h$ such that 
\begin{equation}
\sum_{ n\in \bbZ} | ( \mB_d[V]
\mathbf{g} )_{n} | ^2  \leqslant C \|V_w(x,\cdot)\|_{H^1(\bbR_v)}^2
\sum_{m\in\bbZ} |g_m|^2 ,
\end{equation}
which completes the proof that $\mB_d[V]$ is a uniformly bounded
linear operator on $l^2$. 
\end{proof}
We have proved that the discrete operators of the scheme based on the signularity-free stationary Wigner equation 
is uniformly bounded with respect to the velocity mesh size.  So the numerical solution using the singularity-free scheme 
could be expected to have a better performance. In the next section, we  will validate this by providing some numerical 
examples.

\section{Numerical Examples}
We consider a potential $V(x)$ given as
\[ \begin{array}{l}
     V \left( x \right) = \left\{ \begin{array}{ll}
       0, & \mathrm{if}\ x \notin \left[ -1.5, 1.5 \right],\\
       0.2, & \mathrm{if}\ x \in \left[ 1.5, 1.5 \right].
     \end{array} \right.
   \end{array} \]
It can be used to describe a square potential barrier of length $3$ whose center is
at $x=0$. $0.2$ is the height of the barrier.  We use the interval $[-25,25]$ as  
the computation domain in the  $x$-space , which means a device with length $50$ is simulated. Two contacts are 
put at the two ends, and the inflow boundary conditions are applied. 
 Set $N_x=l/\Delta x$, the grid number in the  $x$-direction. Truncate the vector 
 $\mathbf{f}$ to be finite. In order to be concise, we use the same symbol 
 $\mathbf{f}= \left\{f_n: n=-N_v/2,\cdots, N_v/2-1\right\}\in \mathbb{R}^{N_v}$
 as before to  denote a numerical distribution computed by using the full-discretization scheme.
$N_v$ is the grid number in the $v$-direction. Recall that the mesh size
 in the $v$-direction $\Delta v=2\pi h =\pi/R^h$.
The trapezoidal quadrature rule is used to calculate  the numerical Wigner
 potential 
 \begin{equation} \label{trapzoid}
  V_w(x,v;L_y,\Delta y) =
  -\frac{1}{\pi} \sum_{j=1}^{N_y} D_V(x,j\Delta y) \sin(j \Delta y v) \Delta y,
 \end{equation}
 where $N_y=L_y/\Delta y$. To avoid aliasing error, we choose $L_y<R^h$.
 The elements of $M^{\Theta_d}$ in \eqref{eq:MTheta_d} and   
 $a_m$ in \eqref{eq:a_m} are all obtained by using  \eqref{trapzoid}. 
 For the discretization in the  $x$-space , we use the  2nd-order upwind finite difference
scheme, that is  
\begin{equation}\label{eq:2ndUDS}
\begin{split}
\frac{\partial f(x_i,v)}{\partial x} = 
	\frac{3 f \left( x_i, v \right) 
- 4 f \left(x_{ i - 1}, v \right) + f \left( x_{i - 2}, v \right)}
{2 \Delta x}
, \quad \mathrm{if} \ \ v >0 ,  \\
  \frac{\partial f(x_i,v)}{\partial x} = \frac{- f 
\left( x_{i + 2}, v \right) + 4 f \left( x_{i + 1}, v \right) 
- 3 f \left( x_i, v \right)}{2 \Delta x},
 \quad \mathrm{if} \ \ v<0.  
 \end{split}
 \end{equation}

For convenience, we call  the discretization \eqref{eq:DiscreteProblem}+\eqref{eq:2ndUDS} original 
scheme, and call the discretization \eqref{eq:NewDiscreteProblem}+\eqref{eq:2ndUDS}   improved 
scheme. 

\subsection{Comparison between the original scheme and the improved scheme}
We consider that the electron inflows only from the left contact. 
We will compare the results computed by the original scheme \eqref{eq:DiscreteProblem}+\eqref{eq:2ndUDS}
and our improved scheme \eqref{eq:NewDiscreteProblem}+\eqref{eq:2ndUDS}.
We plot the distribution function as a function of $v$ close to the
left contact shown in \figref{fig1} and at center of the device in \figref{fig2}. 
From the two figures in \figref{fig:compare}, we find that the distributions
are evidently different, especially at the point $v=0$. The numerical distribution function
obtained using the original scheme grows very fast when $v\rightarrow 0$, which reflects 
the effect of the singularity at $v=0$.  Our signularity-free scheme succeeds in solving  the singularity issue.
\begin{figure}
\centering
\subfigure[close to the left of the contact]{
\label{fig1}
\includegraphics[width=7cm]{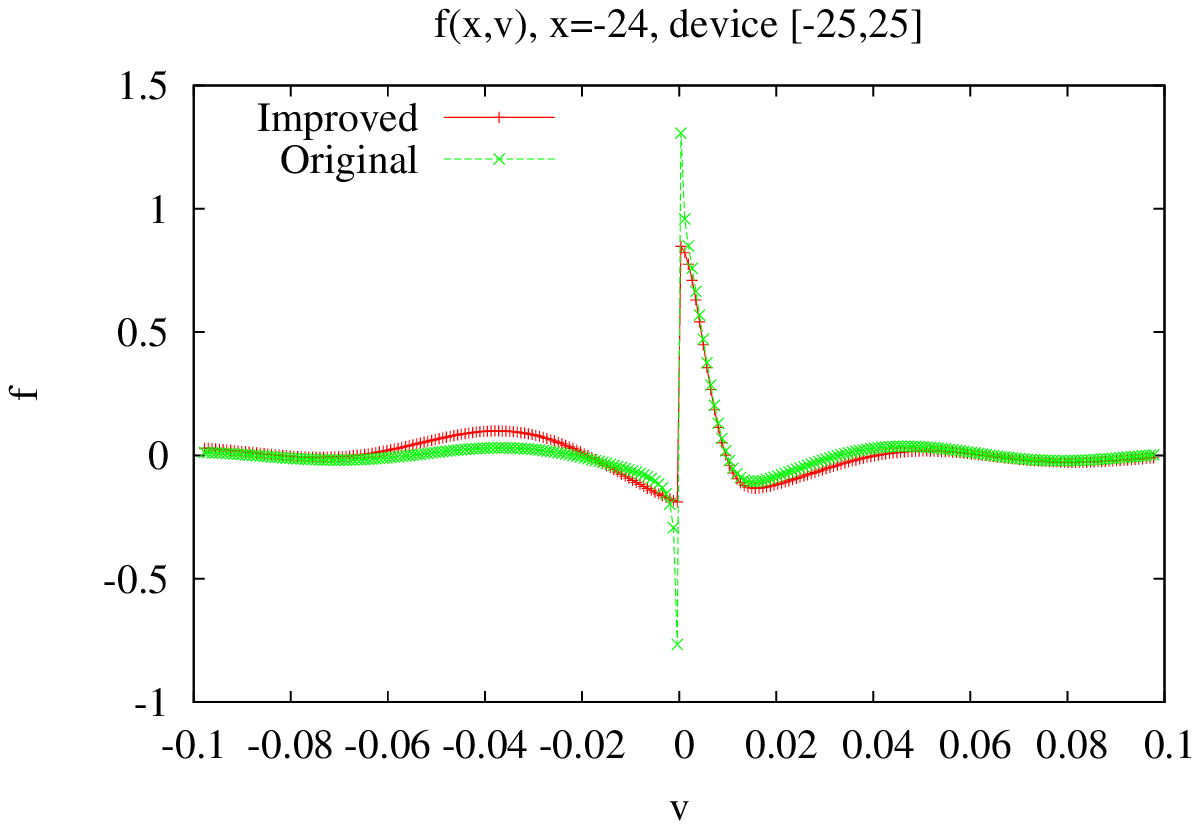}
}
\subfigure[at the center of the device]{
\includegraphics[width=7cm]{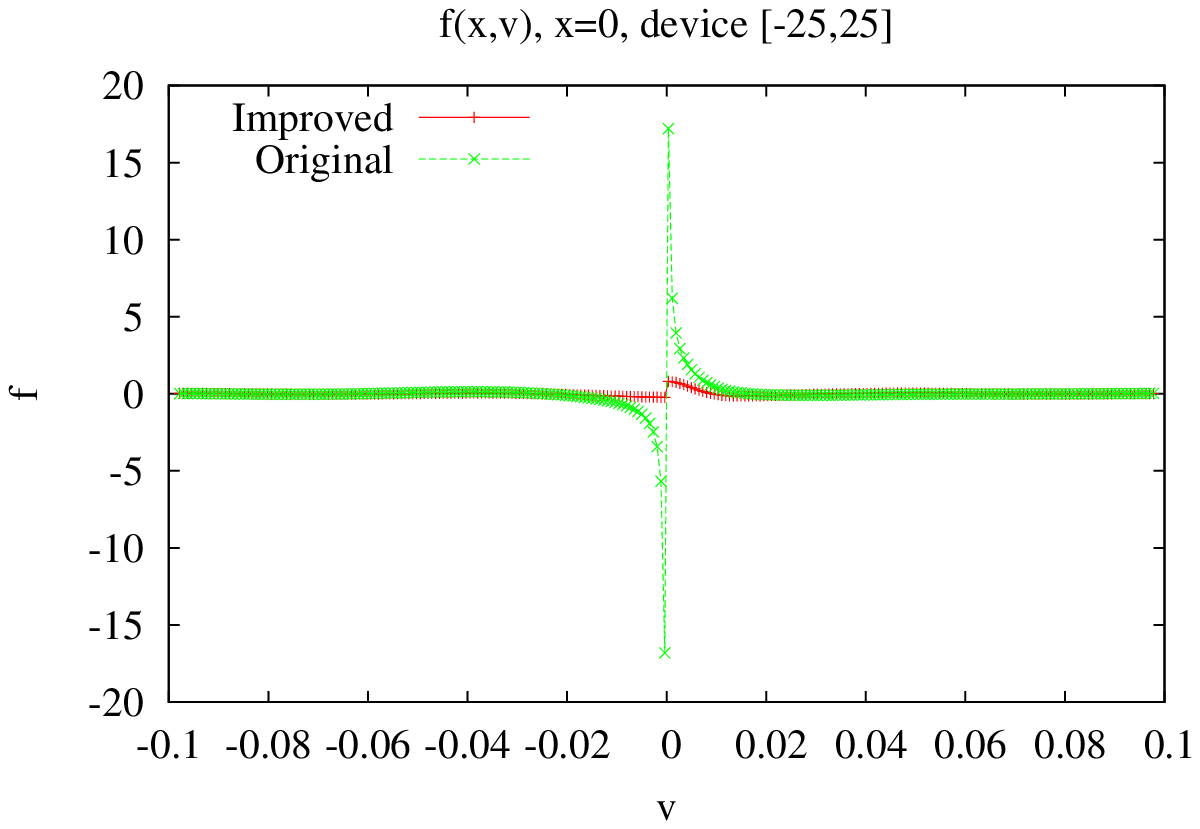}
\label{fig2}
}
\caption{\small  
Device $x\in [-25,25]$. 
The height of the barrier$H$ is $0.2$, and the barrier is put in the center of devices, $[-1.5,1.5]$. 
Inflow only from the left contact, $f_L(v)= \exp(-v^2/0.0002)$ and $f_R(v)=0$.
$N_x=100$. $N_v=128$. $R^h=2048$. $Ly=31$. $\Delta y=0.5$.
\label{fig:compare}} 
\end{figure}

\subsection{Convergence \label{sec:convergence}}
Now we concentrate on studying the convergence with respect to the $x$-mesh size $\Delta x$ and 
with the $v$-mesh size $\Delta v$, respectively. In this example, 
we set $f_L(v)=e^{-(v-0.5\pi)^2/0.25}$, $f_R(v)=0$ as the inflow boundary conditions.

To study the convergence of $v$-direction, we fix $N_x=100$ ( which corresponds to $\Delta x=0.5$). 
We fix the velocity interval $[-\pi,\pi]$, and set $\Delta v= \frac{2\pi}{N_v}$.  The number of velocity points
$N_v=64,128, 256,512,1024$  will be used. Correspondingly, we choose 
$R^h= 32, 64, 128, 256, 512$, which means $\Delta v$ is equal to $2\pi h$ ($h=\frac{1}{2R^h}$). 
 $\Delta y=1.0$ will be used to evaluate the numerical Wigner potential.
 
The $L_2$-norm error is given in the Table \ref{tab:NvError} where the reference
is the solution of the finest mesh ($N_v= 1024$).  
 \begin{table}
 \centering
  \begin{tabular}{|c|c|c|c|c|}
   \hline
    $N_v$& Original &Order & Improved &Order\\
    \hline
    64 & 0.2756 & & 0.05906&\\
    \hline
    128 & 0.2466 &  0.1604 & 0.01446& 2.0301\\
    \hline
    256 & 0.2090 & 0.2386&  0.003473& 2.0577 \\
    \hline
    512 & 0.1505 & 0.4742 &  0.0007513& 2.2090\\
    \hline
  \end{tabular}
  \caption{\small The $L^2$-norm error of the distribution function obtained by using the original scheme and the improved scheme with different numbers of velocity mesh points in
  the fixed velocity interval $[-\pi,\pi]$.\label{tab:NvError}}
 \end{table}

To study the convergence of $x$-direction, we fix the range of velocity $[-\pi,\pi]$, $N_v=256$, 
 $R^h=128$, $L_y=64$. $\Delta y=1.0$. Then we implement the numerical computation by using 
 the orignal scheme and the improved scheme with $N_x=25, 50, 100, 200, 400$, respectively.
 We use the numerical distribution calculated on the finnnest mesh $N_x=400$ as the reference, 
 then the $L^2$-errors and  the convergence orders are shown in Table \ref{tab:NxError}.
 
 \begin{table}
 \centering
  \begin{tabular}{|c|c|c|c|c|}
   \hline
    $N_x$& Original &Order & Improved &Order\\
    \hline
    25 & 0.4208 & & 0.1653&\\
    \hline
    50 & 0.1792 &  1.2322 & 0.0613& 1.4312\\
    \hline
    100 & 0.0623 & 1.5238&  0.0156& 1.9753 \\
    \hline
    200 & 0.0131 & 2.2549 &  0.0030& 2.3590\\
    \hline
  \end{tabular}
  \caption{\small The $L^2$-norm error of the distribution function obtained by using the original scheme and the improved scheme with different numbers of $x$ mesh points in
  the fixed space interval $[-25,25]$.\label{tab:NxError}}
 \end{table}

 \begin{figure}
  \centering
  \subfigure[$N_v-Error$]{
  \includegraphics[width=7cm]{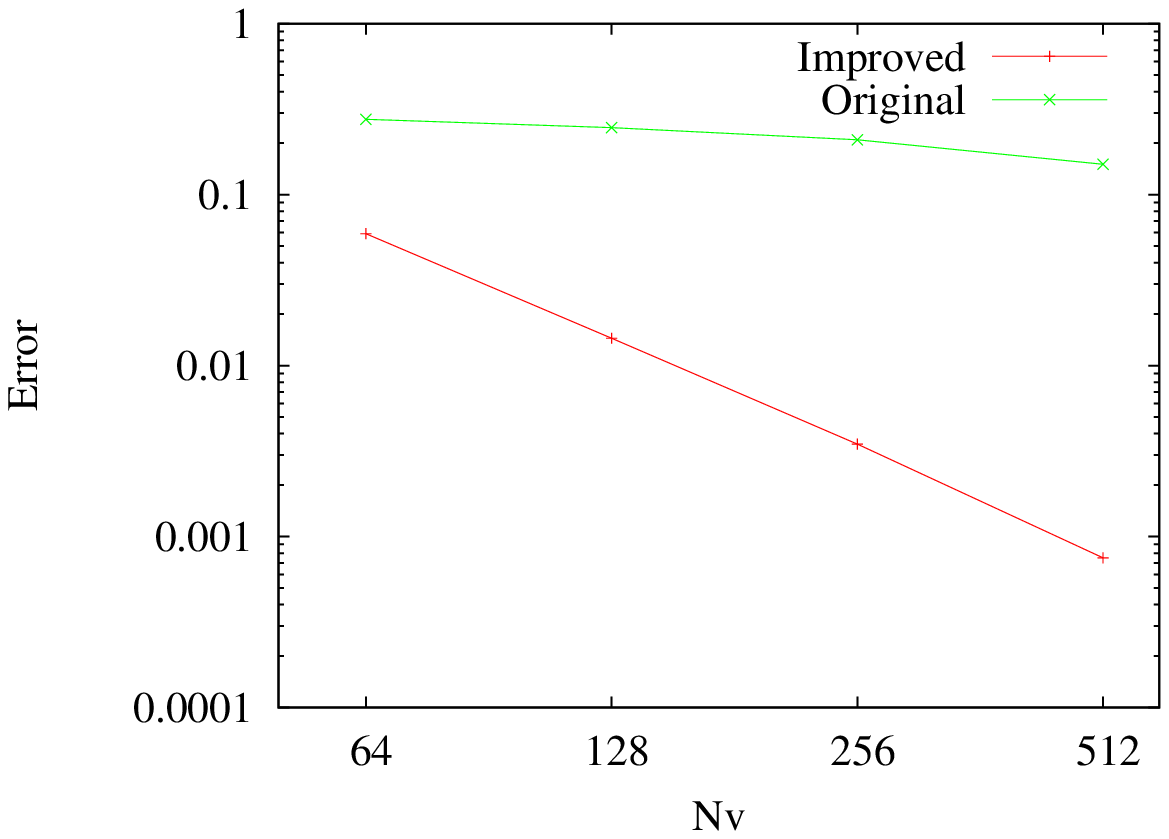}
  }
  \subfigure[$N_x-Error$]{
  \includegraphics[width=7cm]{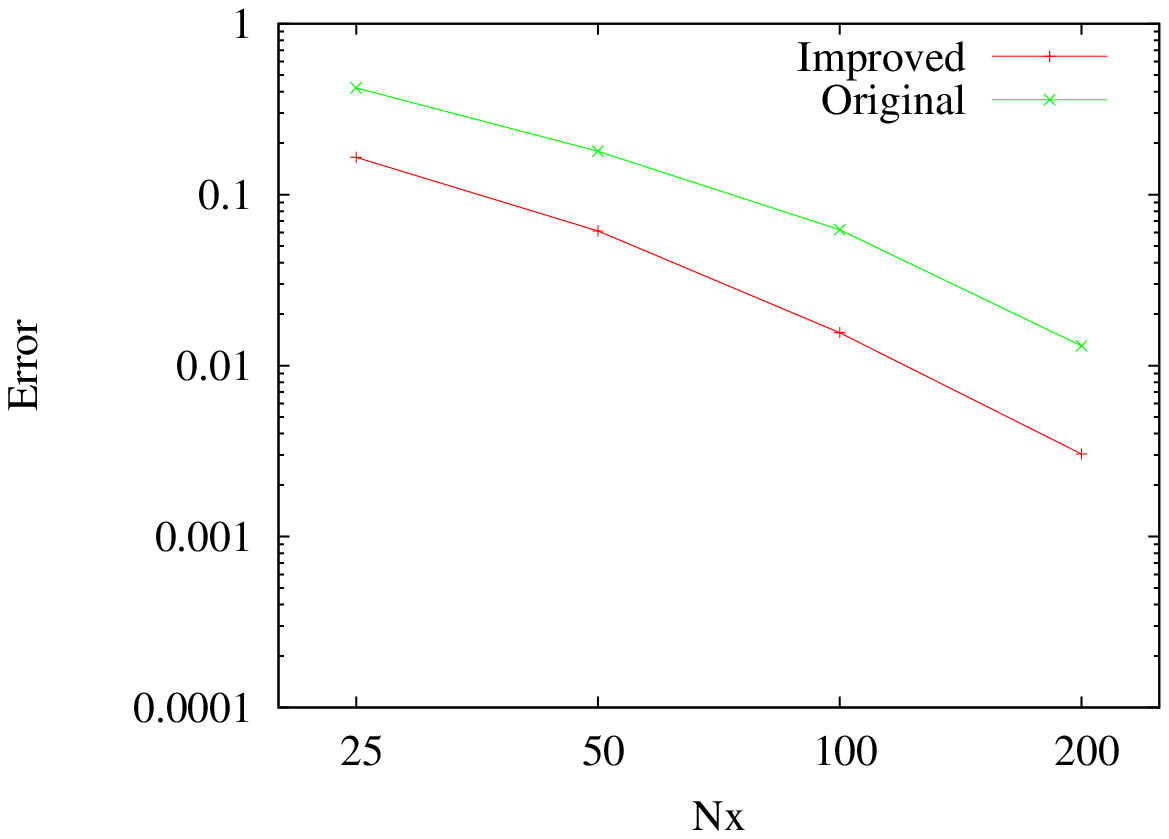}
  }
  \caption{\small Error change with mesh size in $v$-space and $x$-space\label{fig:error} }
 \end{figure}

We can conclude from \figref{fig:error} that improved scheme converges faster than 
original scheme in the $v$-direction, which is mainly contributed to the improved
scheme is based on the equivalent singularity-free Wigner equation. 
The convergence order of improved scheme is 2.0948 while that of original scheme is 0.2853 in $v$-space.  
Original scheme hardly converges for this problem.
$x$-direction, original scheme and improved scheme both gain  convergence. 
The convergence order of the original scheme and improved scheme is 1.6556 and 
1.9272 respectively. This is roughly conformed to the the theoretical analysis of 
second-order upwind scheme.  Therefore, improved scheme is better than original scheme in
convergence.

When we derive the equivalent from of the Wigner equation. We have used an algebra constraint $\eqref{eq:V0}$.
In order to check how well the constraint is sastisfied numerically, we introduce 
\begin{equation}
S(N_v) = \max_i\left\{\sum_{j=-N_v/2}^{N_v/2-1} f(x_i,v_j)V_w(x_i,v_j) \Delta v \right\}
\end{equation}
which is a numerical approximation  of $\max_{x\in[-25,25]}|\int_{\mathbb{R}} V_w(x,-v')f(x,v')\dd v'|$. 
 In Table \ref{tab:Rat}, we list $S(N_v)$ calculated with different $N_v$'s, which shows that 
 $S(N_v)$ decreases to $0$ with a order $1.0$ as $\Delta v \rightarrow 0$  for both the original
scheme and the improved scheme.  This implies that the solution of the stationary Wigner equation 
satisfies the constraction \eqref{eq:V0}, and the explicit use of this property in our improved scheme 
helps in removing singularity  and improving numerical convergence. 

\begin{table}
\centering
 \begin{tabular}{|c|c|c|c|c|c|}
 \hline 
 $N_v$ & 64 &128 & 256 & 512 & 1024 \\
 \hline
  Original & 4.7370e-4 & 2.3550e-4 & 1.1710e-4 & 5.8227e-5 & 2.8945e-5 \\
  \hline
  Improved & 4.2746e-4 & 2.1371e-4 & 1.0685e-5  & 5.3427e-5 & 2.6713e-5 \\
  \hline
 \end{tabular}
\caption{\small $S(N_v)$ with different $N_v$\label{tab:Rat}}
\end{table}

\section{Conclusion}
By using an algebra of the stationary Wignner equation, we have proposed 
a singularity-free scheme, whose numerical convergence with respect
to the velocity mesh size has been validated by numerical experiments. 
We believe that it is the first time that the 
numerical convergence with respect to the velocity mesh size has been obtained for the
stationary Wigner equation with inflow boundary conditions. We will investigate
whether it could be applied in simulation of nano-scale semiconductor devices where
the potential function may not satisfy the condition in Theorem \ref{Thm:BoundofB}.

\section*{Acknowledgements}
This research was supported in part by the NSFC (91434201,91230107,11421101).


\end{document}